\documentclass[12pt]{amsart}
\usepackage{amsmath,amssymb,amsthm,upref,graphicx,mathrsfs}
\usepackage{color}
\usepackage[
colorlinks=true,
linkcolor=blue,
 citecolor=blue,
  urlcolor=blue,
]{hyperref}

\numberwithin{equation}{section}

\newtheorem{theorem}{Theorem}[section]
\newtheorem{prop}[theorem]{Proposition}
\newtheorem{lemma}[theorem]{Lemma}

\theoremstyle{definition}

\newcommand{\Mm}{\mathcal{M}}
\newcommand{\R}{\mathbb{R}}
\newcommand{\Rn}{\mathbb{R}^n}

\let \a=\alpha
\let \O=\Omega
\let \ve=\varepsilon
\let \si=\sigma
\let \d=\delta
\let \la=\lambda
\let \o=\omega
\let \e=\epsilon

\begin{document}
\title[Sharp weighted bounds]
{Sharp weighted bounds for multilinear maximal functions and Calder\'on-Zygmund operators}
\authors
\author[W. Dami\'an]{Wendol\'in Dami\'an}
\address{Wendol\'in Dami\'an\\
Departamento de An\'alisis Matem\'atico, Facultad de Matem\'aticas,
Universidad de Sevilla, 41080 Sevilla, Spain}
\email{wdamian@us.es}
\author[A.K. Lerner]{Andrei K. Lerner}
\address{Andrei K. Lerner\\ Department of Mathematics,
Bar-Ilan University, 52900 Ramat Gan, Israel}
\email{aklerner@netvision.net.il}
\author[C. P\'erez]{Carlos P\'erez}
\address{Carlos P\'erez\\
Departamento de An\'alisis Matem\'atico, Facultad de Matem\'aticas,
Universidad de Sevilla, 41080 Sevilla, Spain}
\email{carlosperez@us.es}


%
%%%%%%%%%%%%%%%%%%%%%%%%%%%%%%%%%%%%%%%%%%%%%%%%%%%%%%%%%%%%%%%%%%%%%%%%%%%%%%%%%%%%%%%%%%%%%%%%%%%%%%%%%%%%%%%%%%%%%%%%%%%%%%%%%%%%%%%%%%%%%%%%%%%
% A B S T R A C T
%%%%%%%%%%%%%%%%%%%%%%%%%%%%%%%%%%%%%%%%%%%%%%%%%%%%%%%%%%%%%%%%%%%%%%%%%%%%%%%%%%%%%%%%%%%%%%%%%%%%%%%%%%%%%%%%%%%%%%%%%%%%%%%%%%%%%%%%%%%%%%%%%%%
%
\begin{abstract}
In this paper we prove some sharp weighted norm inequalities for the multi(sub)linear maximal function $\Mm$ introduced in \cite{LOPTT}
and for multilinear Calder\'on-Zygmund operators.
In particular we obtain a sharp mixed ``$A_p-A_{\infty}$" bound for $\Mm$, some partial results related to a Buckley-type estimate for $\Mm$,
and a sufficient condition for the boundedness of $\Mm$ between weighted $L^p$ spaces with different weights taking into account the precise bounds.

Next we get a bound for multilinear Calder\'on-Zygmund operators in terms of dyadic positive multilinear operators in the spirit of the recent work \cite{Ler1}.
Then we obtain a multilinear version of the ``$A_2$ conjecture". Several open problems are posed.
\end{abstract}
\maketitle
%\tableofcontents
%
%%%%%%%%%%%%%%%%%%%%%%%%%%%%%%%%%%%%%%%%%%%%%%%%%%%%%%%%%%%%%%%%%%%%%%%%%%%%%%%%%%%%%%%%%%%%%%%%%%%%%%%%%%%%%%%%%%%%%%%%%%%%%%%%%%%%%%%%%%%%%%%%%%%
%   I N T R O D U C T I O N
%%%%%%%%%%%%%%%%%%%%%%%%%%%%%%%%%%%%%%%%%%%%%%%%%%%%%%%%%%%%%%%%%%%%%%%%%%%%%%%%%%%%%%%%%%%%%%%%%%%%%%%%%%%%%%%%%%%%%%%%%%%%%%%%%%%%%%%%%%%%%%%%%%%
%
\section{Introduction}

The modern theory of weighted norm inequalities for some of the main operators in Harmonic Analysis originated in the beginning of the 70's in the works by
R. Hunt, B. Muckenhoupt, R. Wheeden, R. Coifman and C. Fefferman. In particular, it was
realized that the key role in this theory is played by the so-called $A_p$ condition.
Much later, the question about the sharp dependence of the $L^p(w)$ operator norm in term of the the $A_p$ constant or     characteristic of the weight appeared. First, for the Hardy-Littlewood maximal operator this problem was solved by S. Buckley \cite{Bu}.

It turned out that for singular integrals the question is much more complicated. In \cite{Pet1}, S. Petermichl solved it
for the Hilbert transform. Recently, T. Hyt\"onen \cite{Hyt} gave a complete solution for general Calder\'on-Zygmund operators
solving the so-called $A_2$ conjecture.
A bit later  this result was improved in \cite{HP} in the case $p=2$ and  for general $p$ in \cite{HL}. A further improvement was obtained in \cite{HLP} where a non-probablistic proof was found together with a $q$-variation estimate.  We refer to these papers for a more detailed history and for some other closely related results like \eqref{hype}.

The aim of this paper is to give some multilinear analogues of the above mentioned results in the spirit of the theory of multiple
weights developed recently in \cite{LOPTT}. We introduce some notation. Given $\vec f=(f_1,\dots,f_m)$, we define the multi(sub)linear
maximal operator $\mathcal M$ by
$$\mathcal M(\vec f\,)(x)=\sup_{Q\ni x}\prod_{i=1}^m\frac{1}{|Q|}\int_Q|f_i(y_i)|dy_i,$$
where the supremum is taken over all cubes $Q$ containing $x$. It is shown in that paper that this operator controls in several ways the class of multilinear Calder\'on-Zygmund operators (see section \ref{MCZO}).  A particular instance of this intimate relationship is the class of weights characterizing the weighted $L^p$ spaces for which both operators are bounded. To define this class of weights we let  $\vec w=(w_1,\dots,w_m)$ and $\vec P=(p_1,\dots, p_m)$.
Set $\frac{1}{p}=\frac{1}{p_1}+\dots+\frac{1}{p_m}$ and $\nu_{\vec w}=\prod_{i=1}^mw_i^{p/p_i}$. We say that $\vec w$ satisfies the
$A_{\vec P}$ condition if
$$[\vec w]_{A_{\vec P}}=\sup_{Q}\Big(\frac{1}{|Q|}\int_Q\nu_{\vec
w}\Big)\prod_{i=1}^m\Big(\frac{1}{|Q|}\int_Q
w_i^{1-p'_i}\Big)^{p/p'_i}<\infty.
$$
It is easy to see that in the linear case (that is, if $m=1$) $[\vec w]_{A_{\vec P}}=[w]_{A_p}$ is the usual $A_p$ constant.
In \cite{LOPTT} the following multilinear extension of the Muckenhoupt $A_p$ theorem for the maximal function was obtained:  the inequality
\begin{equation}\label{strong}
\|{\mathcal M}(\vec f\,)\|_{L^{p}(\nu_{\vec w})}\le C \prod_{i=1}^m\|f_i\|_{L^{p_i}(w_i)}
\end{equation}
holds for every  $\vec f$ if and only if $\vec w$ satisfies the $A_{\vec P}$ condition.

The first question we are going to study is the question about the sharp dependence of $C$ in (\ref{strong}).
In the standard situation, namely when $m=1$, two different types of sharp weighted inequalities for the Hardy-Littlewood maximal operator $M$ are known.
First, Buckley's theorem \cite{Bu} says that for any $1<p<\infty$,
\begin{equation}\label{buckley}
\|M\|_{L^p(w)}\le C_{n,p}[w]_{A_p}^{\frac{1}{p-1}}.
\end{equation}

Until very recently this result was thought to be sharp since the exponent of
$[w]_{A_p}$ cannot be improved. However, recently T. Hyt\"onen and the third author \cite{HP} (see also \cite{HPR} for another proof)  showed that (\ref{buckley}) can be improved in the following way. Define the $A_{\infty}$ constant of $w$ by
$$[w]_{A_{\infty}}=\sup_Q\frac{1}{w(Q)}\int_QM(w\chi_Q).$$
This definition goes back to the characterization of the $A_{\infty}$ class of weights given by N. Fujii \cite{Fu} (see also M. Wilson \cite{Wil1}). It is not difficult to show that $[w]_{A_{\infty}}\le c_{n}[w]_{A_r}$ for any $r\ge 1$.
It was shown in \cite{HP} that
\begin{equation}\label{hype}
\|M\|_{L^p(w)}\le C_{n,p}([w]_{A_p}[\sigma]_{A_{\infty}})^{1/p},
\end{equation}
where $\si=w^{1-p'}$. This estimate implies (\ref{buckley}) since $[\sigma]_{A_{\infty}}\le c[\sigma]_{A_{p'}}=c[w]_{A_p}^{\frac{1}{p-1}}$, and therefore
$([w]_{A_p}[\sigma]_{A_{\infty}})^{1/p}\le c[w]_{A_p}^{\frac{1}{p-1}}$.

Our multilinear results related to (\ref{buckley}) and (\ref{hype}) are somewhat surprising. They show that contrary to the  case $m=1,$ multilinear versions of (\ref{buckley}) and (\ref{hype}) are independent of each other.
We are able to get a full analogue of (\ref{hype}). However such an analogue does not yield an expected full analogue of (\ref{buckley}).
Our first main result is the following.

\begin{theorem}\label{Main_thm}
Let $1<p_i<\infty,i=1,\dots,m$ and $\frac{1}{p}=\frac{1}{p_1}+\ldots+\frac{1}{p_m}$. Then the inequality
\begin{equation}\label{OneWeightBestConstant}
\|{\mathcal M}(\vec f\,)\|_{L^{p}(\nu_{\vec w})}\le C_{n,m, \vec P}\,[\vec w]_{A_{\vec P}}^{\frac{1}{p}}\,
\prod_{i=1}^m ([\sigma_i]_{A_{\infty}})^{\frac{1}{p_i}} \, \prod_{i=1}^m\|f_i\|_{L^{p_i}(w_i)}
\end{equation}
holds if $\vec w\in A_{\vec P}$, where $\sigma_i=w_i^{1-p_i'}$, $i=1,\ldots,m$.  Furthermore the exponents are sharp in the sense that they
cannot be replaced by smaller ones.
\end{theorem}

Even though the result of this  theorem is sharp it would be of interest to find an extension of Buckley's estimate \eqref{buckley}. However this task seems to be more complicated. Currently we can get only several partial results expressed in the following theorem.

\begin{theorem}\label{multbuckley}
Let $1< p_i<\infty,i=1,\dots,m$ and $\frac{1}{p}=\frac{1}{p_1}+\ldots+\frac{1}{p_m}$.
Denote by $\a=\a(p_1,\dots,p_m)$ the best possible power in
\begin{equation}\label{sharpbuck}
\|{\mathcal M}(\vec f\,)\|_{L^{p}(\nu_{\vec w})}\le C_{n,m,p}\, [\vec w]_{A_{\vec P}}^{\a}\,\prod_{i=1}^m\|f_i\|_{L^{p_i}(w_i)}.
\end{equation}
Then we have the following results:
\begin{enumerate}
\renewcommand{\labelenumi}{(\roman{enumi})}
\item for all $1<p_1,\dots,p_m<\infty$, $\frac{m}{mp-1}\le \a\le \frac{1}{p}\Big(1+\sum_{i=1}^m\frac{1}{p_i-1}\Big);$
\item if $p_1=p_2=\dots=p_m=r>1$,\, then \,$\a=\frac{m}{r-1}$.
\end{enumerate}

\end{theorem}

It is easy to see that the upper and lower bounds for $\a$ in (i) coincide if $m=1$.
The upper bound for $\a$ in (i) is a corollary of Theorem \ref{Main_thm} after an application of Lemma \ref{technical_lemma} (see section 3). However, in the case of (ii) $\a$ coincides with the lower bound in (i).
This says that if $m\ge 2$, then the upper bound in (i) is not sharp, in general. Hence, contrary to the linear situation, (\ref{OneWeightBestConstant}) cannot be used in order to get a sharp bound in terms of $[\vec w]_{A_{\vec P}}$. Perhaps, the explanation of this is that the right-hand side of (\ref{OneWeightBestConstant})
involves $m+1$ independent suprema while the definition of $[\vec w]_{A_{\vec P}}$ involves only one supremum or else  Lemma \ref{technical_lemma} is not sharp. Resuming, the problem of finding the sharp $\a$ in (\ref{sharpbuck}) remains open except the case considered in (ii).

We also give a sufficient condition for the ``two-weighted" boundedness of ${\mathcal M}$ with precise bounds generalizing the corresponding linear result from \cite{P1} and its multilinear counterpart in \cite{Mo1}. Let $X$ be a Banach function space.
By $X'$ we denote the associate space to $X$. Given a cube $Q$, define the $X$-average of $f$ over $Q$ and the maximal operator $M_X$ by
$$\|f\|_{X,Q}=\|\tau_{\ell_Q}(f\chi_Q)\|_X,\quad M_Xf(x)=\sup_{Q\ni x}\|f\|_{X,Q},$$
where $\ell_Q$ denotes the side length of $Q$ and where \,$\tau_{\delta}f=f(\delta x)$,\, $\delta>0, x\in \Rn$.

\begin{theorem}\label{TwoWeights_Thm}
Let $1<p_i<\infty,i=1,\dots,m$ and $\frac{1}{p}=\frac{1}{p_1}+\ldots+\frac{1}{p_m}$.
Let $X_i$ be a Banach function space such that $M_{X_i'}$ is bounded on $L^{p_i}(\Rn)$.
Let $u$ and $v_1,\dots,v_m$ be the weights satisfying
$$
K=\sup_Q \Big(\frac{u(Q)}{|Q|}\Big)^{\frac{1}{p}}\prod_{i=1}^m\|v_i^{-1/p_i}\|_{X_i,Q}<\infty.
$$
Then
$$
\|{\mathcal M}(\vec f\,)\|_{L^p(u)}\le C_{n,m}K\prod_{i=1}^{m}\|M_{X_i'}\|_{L^{p_i}(\Rn)}\|f_i\|_{L^{p_i}(v_i)}.
$$
\end{theorem}

This result can be seen as a  two weight version  of \eqref{OneWeightBestConstant}
when considering function spaces $X$ given by $X=L^{rp'}$ for $1<p,r<\infty$ so that
$$\|M_{X'}\|_{L^{p}(\R^n)}  = \|M_{ (rp')'  }\|_{L^{p}(\R^n)} \approx (r')^{1/p}. $$
Another interesting example is given when considering the Orlicz space space
\,$X=L_B$\, where $B$ is a Young function for which $\|M_{X'}\|_{L^{p}(\R^n)}  = \|M_{\bar{B}}\|_{L^{p}(\R^n)}$ is finite.  In particular if \,$B(t)=t^{p'} \,(\log( e +t))^{p'-1+\delta}$\,, \,$\delta>0$,\,$1<p<\infty$\, it follows from \cite{P1}  that
$$    \|M_{X'}\|_{L^{p}(\R^n)}  = \|M_{\bar{B}}\|_{L^{p}(\R^n)} \approx (\frac{1}{\delta})^{1/p}.$$
%

%\comment{  We could say more but Kabe has developed this within the context of fractional multilinears in \cite{Mo1} but  he does not care about sharp constants (it was part of his thesis) , the main theme here is that last theorem can be seen as extension of Theorem \ref{Main_thm}. }

We turn now to some multilinear analogues of the sharp weighted results for singular integrals. First, we recall some linear results. Let $T$ be a Calder\'on-Zygmund operator. Then it was proved by T. Hyt\"onen \cite{Hyt} in full generality that
\begin{equation}\label{hyta2}
\|T\|_{L^p(w)}\le C_{T,n,p}[w]_{A_p}^{\max(1,\frac{1}{p-1})}\quad (1<p<\infty).
\end{equation}
Observe that it suffices to prove (\ref{hyta2}) in the case $p=2$; then for any other $p$ the result follows by the sharp extrapolation theorem. Note also that (\ref{hyta2}) in the case
$p=2$ was usually referred as the $A_2$ conjecture. In two recent papers \cite{Ler1,Ler2} a different proof of (\ref{hyta2}) was found by the second author.
This proof shows that there is an intimate relationship between the (continuous) singular integral $T$ and some very special  dyadic type operators. These operators are defined by means of the concept  of  sparseness.  Given a sparse family
${\mathcal S}=\{Q_j^k\}$ of cubes from a dyadic grid ${\mathscr D}$, (these notions will be defined in Section \ref{dyadicstuff}) we consider the operator ${\mathcal A}_{{\mathcal S}, {\mathscr D}}$ defined by
$$
{\mathcal A}_{{\mathcal S},{\mathscr D}}f(x)=\sum_{j,k}f_{Q_j^k}\chi_{Q_j^k}(x).
$$
It was proved in \cite{Ler1,Ler2} that for any Banach function space $X$,
\begin{equation}\label{tbya}
\|Tf\|_{X}\le c_{T,n}\sup_{{\mathcal S}, {\mathscr D}}\|{\mathcal A}_{{\mathcal S},{\mathscr D}}|f|\|_{X}.
\end{equation}
A rather simple argument found in \cite{CMP} shows that $\|{\mathcal A}_{{\mathcal S},{\mathscr D}}\|_{L^2(w)}\le c[w]_{A_2}$. Therefore if $X=L^2(w)$, then
(\ref{tbya}) implies (\ref{hyta2}) for $p=2$.

%{\color{red} On the other hand a very interesting argument from \cite{Mo2} shows that  the extrapolation theorem  can be completely avoided. }

The second main result of this paper is an extension of (\ref{tbya}) to the multilinear setting. Here, the dyadic operator that appears naturally is a multilinear version of the operator
${\mathcal A}_{{\mathscr{D}},{\mathcal S}}$ given by
$${\mathcal A}_{{\mathscr{D}},{\mathcal S}}(\vec f\,)(x)=\sum_{j,k}\Big(\prod_{i=1}^m(f_i)_{Q_j^k}\Big)\chi_{Q_j^k}(x).$$
Given $\vec f=(f_1,\dots,f_m)$, denote $\vec{|f|}=(|f_1|,\dots,|f_m|)$.

\begin{theorem}\label{multcz}
Let $T(\vec f\,)$ be a multilinear Calder\'on-Zygmund operator and
let $X$ be a Banach function space over ${\mathbb R}^n$ equipped with Lebesgue measure.
Then, for any appropriate $\vec f$,
$$
\|T(\vec f\,)\|_{X}\le c_{T,m,n}\sup_{{\mathscr{D}},{\mathcal S}}\|{\mathcal A}_{{\mathscr{D}},{\mathcal S}}(\vec{|f|})\|_{X},
$$
where the supremum is taken over arbitrary dyadic grids ${\mathscr{D}}$ and sparse families ${\mathcal S}\in {\mathscr{D}}$.
\end{theorem}

Similarly to the linear case, we would like to apply this result when $X=L^{p}(\nu_{\vec w})$. However, the exponent $p$ is allowed to be smaller than one, to be more precise  \,$1/m<p<\infty$ (see section \ref{MCZO}). Therefore, if $1/m<p<1$, Theorem \ref{multcz} cannot be applied since in this case the space $L^{p}(\nu_{\vec w})$ is not a Banach space.
This raises an interesting question whether the condition in Theorem \ref{multcz} that $X$ is a Banach space can be relaxed until that $X$ is a quasi-Banach space.
It is natural to consider first this question in the linear situation. Observe that in the current proof of (\ref{tbya}) the fact that $X$ is a Banach space was essential.

The third main result of this paper can be seen as a multilinear version of the $A_2$ conjecture (\ref{hyta2}). Indeed, one of the main results obtained in  \cite{LOPTT} is that if $\vec w\in A_{\vec P}$, then an analogue of (\ref{strong}) holds with $T(\vec f\,)$ instead of ${\mathcal M}(\vec f\,)$. Hence, the question about the sharp
dependence on $[\vec w]_{A_{\vec P}}$ in the corresponding inequality is quite natural.
It is interesting that contrary to the linear situation where (\ref{tbya}) implies the $A_2$ conjecture, we are currently able to apply Theorem \ref{multcz} only in one particular case
being the content of the following theorem.

\begin{theorem}\label{mult} Let $T(\vec f\,)$ be a multilinear Calder\'on-Zygmund operator. Assume that $p_1=p_2=\dots=p_m=m+1$. Then
\begin{equation}\label{mucz}
\|T(\vec f\,)\|_{L^p(\nu_{\vec w})}\le C_{T,m,n}[\vec w]_{A_{\vec P}}\prod_{i=1}^m\|f_i\|_{L^{p_i}(w_i)}.
\end{equation}
\end{theorem}

Several remarks about this result are in order. As we already mentioned, Theorem \ref{mult} can be regarded as a multilinear ``$A_2$ conjecture". However, it is natural to ask
how to extend it to all $1<p_i<\infty$. This leads to several interesting problems. As we explained above, (\ref{hyta2}) can be obtained from the case $p=2$ by the sharp version of the extrapolation theorem of Rubio de Francia obtained in \cite{DGPP} (see also \cite{D} or \cite{CMP2} for different and simpler proofs). It would be very desirable to get a multilinear analogue of this result. Having such an analogue, inequality (\ref{mucz}) probably would be a starting point to extrapolate from.

Observe, however, that (\ref{hyta2}) can be proved also without the use of extrapolation. Indeed, one can easily prove (\ref{hyta2}) with ${\mathcal A}_{{\mathscr{D}},{\mathcal S}}$ instead of $T$ for all $1<p<\infty$ (as it was done in \cite{CMP} for $p=2$), and then apply (\ref{tbya}) with $X=L^p(w)$.
The proof of (\ref{hyta2}) for ${\mathcal A}_{{\mathscr{D}},{\mathcal S}}$ is very close in spirit to the proof of Buckley's inequality (\ref{buckley}) found in
\cite{Ler3}. Hence, it is natural to ask whether it is possible to find a similar proof for a multilinear version of
${\mathcal A}_{{\mathscr{D}},{\mathcal S}}$. But this leads to a problem of finding $\a$ in Theorem \ref{multbuckley} by the method in \cite{Ler3}.
Part (ii) of Theorem \ref{multbuckley} indeed obtained by an adaptation of this method. However how to do that in the case of different $p_i$ is not clear.

The paper is organized as follows. Some preliminaries are contained in Section 2. In Section 3 we prove all theorems related to ${\mathcal M}$.
Section 4 is devoted to the proof of Theorem \ref{multcz}. Finally, Theorem \ref{mult} is proved in Section 5.

\section{Preliminaries}
\subsection{Dyadic grids}\label{dyadicstuff}
Recall that the standard dyadic grid in ${\mathbb R}^n$ consists of the cubes
$$2^{-k}([0,1)^n+j),\quad k\in{\mathbb Z}, j\in{\mathbb Z}^n.$$
Denote the standard grid by ${\mathcal D}$.

By a {\it general dyadic grid} ${\mathscr{D}}$ we mean a collection of
cubes with the following properties: (i)
for any $Q\in {\mathscr{D}}$ its sidelength $\ell_Q$ is of the form
$2^k, k\in {\mathbb Z}$; (ii) $Q\cap R\in\{Q,R,\emptyset\}$ for any $Q,R\in {\mathscr{D}}$;
(iii) the cubes of a fixed sidelength $2^k$ form a partition of ${\mathbb
R}^n$.

We say that $\{Q_j^k\}$ is a {\it sparse family} of cubes if:
(i)~the cubes $Q_j^k$ are disjoint in $j$, with $k$ fixed;
(ii) if $\Omega_k=\cup_jQ_j^k$, then $\Omega_{k+1}\subset~\Omega_k$;
(iii) $|\Omega_{k+1}\cap Q_j^k|\le \frac{1}{2}|Q_j^k|$.

With each sparse family $\{Q_j^k\}$ we associate the sets $E_j^k=Q_j^k\setminus \O_{k+1}$.
Observe that the sets $E_j^k$ are pairwise  disjoint and $|Q_j^k|\le 2|E_j^k|$.

Given a cube $Q_0$, denote by ${\mathcal D}(Q_0)$ the set of all
dyadic cubes with respect to $Q_0$, that is, the cubes from ${\mathcal D}(Q_0)$ are formed
by repeated subdivision of $Q_0$ and each of its descendants into $2^n$ congruent subcubes.
Observe that if $Q_0\in {\mathscr{D}}$, then each cube from ${\mathcal D}(Q_0)$ will also
belong to ${\mathscr{D}}$.

We will use the following proposition from \cite{HP}.

\begin{prop}\label{prhp} There are $2^n$ dyadic grids ${\mathscr{D}}_{\a}$ such that for any cube $Q\subset {\mathbb R}^n$ there exists a cube $Q_{\a}\in {\mathscr{D}}_{\a}$
such that $Q\subset Q_{\a}$ and $\ell_{Q_{\a}}\le 6\ell_Q$.
\end{prop}

\begin{lemma}\label{appr}
For any non-negative integrable $f_i,i=1,\dots,m$, there exist sparse families ${\mathcal S}_{\a}\in {\mathscr{D}}_{\a}$ such that for all $x\in {\mathbb R}^n$,
$${\mathcal M}(\vec f\,)(x)\le (2\cdot 12^n)^{m}\sum_{\a=1}^{2^n}
{\mathcal A}_{{\mathscr{D}_{\a}},{\mathcal S}_{\a}}(\vec f\,)(x).$$
\end{lemma}

\begin{proof} First, by Proposition \ref{prhp},
\begin{equation}\label{estd}
{\mathcal M}(\vec f\,)(x)\le 6^{mn}\sum_{\a=1}^{2^n}{\mathcal M}^{{\mathscr{D}}_{\a}}(\vec f\,)(x).
\end{equation}

Consider ${\mathcal M}^d(\vec f\,)$ taken with respect to the standard dyadic grid. We will use
exactly the same argument as in the Calder\'on-Zygmund decomposition. For $c_n$ which will be specified below and for
$k\in {\mathbb Z}$ consider the sets
$$\Omega_k=\{x\in {\mathbb R}^n: {\mathcal M}^d(\vec f\,)(x)>c_n^k\}.$$
Then we have that $\Omega_k=\cup_jQ_j^k$, where the cubes $Q_j^k$ are pairwise disjoint with $k$ fixed, and
$$c_n^k<\prod_{i=1}^m(f_i)_{Q_j^k}\le 2^{mn}c_n^k.$$
From this and from H\"older's inequality,
\begin{eqnarray*}
|Q_j^k\cap \Omega_{k+1}|&=&\sum_{Q_l^{k+1}\subset Q_j^k}|Q_l^{k+1}|\\
&<&c_n^{-\frac{k+1}{m}}\sum_{Q_l^{k+1}\subset Q_j^k}\prod_{i=1}^m\Big(\int_{Q_l^{k+1}}f_i\Big)^{1/m}\\
&\le& c_n^{-\frac{k+1}{m}}\prod_{i=1}^m\Big(\int_{Q_j^{k}}f_i\Big)^{1/m}\le 2^nc_n^{-1/m}|Q_j^k|.
\end{eqnarray*}
Hence, taking $c_n=2^{m(n+1)}$, we obtain that the family $\{Q_j^k\}$ is sparse, and
$${\mathcal M}^d(\vec f\,)(x)\le 2^{m(n+1)}{\mathcal A}_{{\mathcal{D}},{\mathcal S}}(\vec f\,)(x).$$

Applying the same argument to each ${\mathcal M}^{{\mathscr{D}}_{\a}}(\vec f\,)$ and using (\ref{estd}),
we get the statement of the lemma.
\end{proof}

\subsection{Multilinear Calder\'on-Zygmund operators} \label{MCZO}
Let $T$ be a multilinear operator initially defined on the $m$-fold
product of Schwartz spaces and taking values into the space of
tempered distributions,
$$T:S(\Rn)\times\dots\times S({\mathbb R}^n)\to
S'({\mathbb R}^n).$$
We say that $T$ is an $m$-linear
Calder\'on-Zygmund operator if, for some $1\le q_j<\infty$, it
extends to a bounded multilinear operator from
$L^{q_1}\times\dots\times L^{q_m}$ to $L^q$, where
$\frac1q=\frac{1}{q_1}+\dots+\frac{1}{q_m}$, and if
 there exists a function
$K$, defined
 off the diagonal $x=y_1=\dots=y_m$ in $({\mathbb
R}^n)^{m+1}$,  satisfying
$$T(f_1,\dots,f_m)(x) =\int\limits_{({\mathbb R}^{n})^m}
K(x,y_1,\dots,y_m)f_1(y_1)\dots f_m(y_m)\,dy_1\dots dy_m
$$
for all $x \not\in \cap_{j=1}^{m }\,\,\text{supp}\,\, f_j$,
$$|K(y_0,y_1,\dots,y_m)|\le \frac{A}{
\Bigl(\sum\limits_{k,l=0}^m|y_k-y_l|\Bigr)^{mn}},
$$
and
$$
|K(y_0,\dots,y_j,\dots,y_m)-K(y_0,\dots,y_j',\dots,y_m)|
\le\frac{A|y_j-y_j'|^{\e}}{\Bigl(
\sum\limits_{k,l=0}^m|y_k-y_l|\Bigr)^{mn+\e}},
$$
for some $\e>0$ and all $0\le j\le m$, whenever $\displaystyle|y_j-y_j'|\le
\frac{1}{2}\max_{0\le k\le
m}|y_j-y_k|$. %

It was shown in \cite{GT1} that if
$\frac{1}{r_1}+\dots+\frac{1}{r_m}=\frac{1}{r}$, then an $m$-linear
Calder\'on-Zygmund operator satisfies
\begin{equation*}
T: L^{r_1}(\Rn)\times\dots\times L^{r_m}(\Rn) \to L^{r}(\Rn)
\end{equation*}
 when $1<r_j<\infty$ for all $j=1,\cdots, m$. Similarly if $1\leq r_j \leq \infty$ for all $j=1,\cdots, m$,  we have
\begin{equation}\label{weaktype}
T:L^{1}(\Rn)\times \dots \times L^{1}(\Rn)\to L^{1/m,\infty}(\Rn)
\end{equation}

\subsection{A ``local mean oscillation decomposition"}
The non-increasing rearrangement of a measurable function $f$ on ${\mathbb R}^n$ is defined by
$$f^*(t)=\inf\{\a>0:|\{x\in {\mathbb R}^n:|f(x)|<\a\}|<t\}\quad(0<t<\infty).$$

Given a measurable function $f$ on ${\mathbb R}^n$ and a cube $Q$,
the local mean oscillation of $f$ on $Q$ is defined by
$$\o_{\la}(f;Q)=\inf_{c\in {\mathbb R}}
\big((f-c)\chi_{Q}\big)^*\big(\la|Q|\big)\quad(0<\la<1).$$

By a median value of $f$ over $Q$ we mean a possibly nonunique, real
number $m_f(Q)$ such that
$$\max\big(|\{x\in Q: f(x)>m_f(Q)\}|,|\{x\in Q: f(x)<m_f(Q)\}|\big)\le |Q|/2.$$

It is easy to see that the set of all median values of $f$ is either one point or the closed interval. In the latter case we will assume for
the definiteness that $m_f(Q)$ is the {\it maximal} median value. Observe that it follows from the definitions that
\begin{equation}\label{pro1}
|m_f(Q)|\le (f\chi_Q)^*(|Q|/2).
\end{equation}

Given a cube $Q_0$, the dyadic local sharp maximal
function $m^{\#,d}_{\la;Q_0}f$ is defined by
$$m^{\#,d}_{\la;Q_0}f(x)=\sup_{x\in Q'\in
{\mathcal D}(Q_0)}\o_{\la}(f;Q').$$

The following theorem was proved in \cite{Ler1} (its very similar version can be found in \cite{L1}).

\begin{theorem}\label{decom1} Let $f$ be a measurable function on
${\mathbb R}^n$ and let $Q_0$ be a fixed cube. Then there exists a
(possibly empty) sparse family of cubes $Q_j^k\in {\mathcal D}(Q_0)$
such that for a.e. $x\in Q_0$,
$$
|f(x)-m_f(Q_0)|\le
4m_{\frac{1}{2^{n+2}};Q_0}^{\#,d}f(x)+2\sum_{k,j}
\o_{\frac{1}{2^{n+2}}}(f;Q_j^k)\chi_{Q_j^k}(x).
$$
\end{theorem}

\subsection{Banach function spaces}
For a general account of Banach function spaces we refer to
\cite[Ch. 1]{BS}. We mention only several notions which will be
used below.

The associate space $X'$ consists of measurable functions $f$ for which
$$\|f\|_{X'}=\sup_{\|g\|_{X}\le 1}\int_{{\mathbb R}^n}|f(x)g(x)|dx<\infty.$$
This definition implies the following H\"older inequality:
\begin{equation}\label{hol}
\int_{{\mathbb R}^n}|f(x)g(x)|dx\le \|f\|_{X}\|g\|_{X'}.
\end{equation}
Further \cite[p. 13]{BS},
\begin{equation}\label{dua}
\|f\|_{X}=\sup_{\|g\|_{X'}=1}\int_{{\mathbb R}^n}|f(x)g(x)|dx.
\end{equation}

%By Fatou's lemma \cite[p. 5]{BS}, if $f_n\to f$ a.e., and if
%$\displaystyle\liminf_{n\to\infty}\|f_n\|_{X}<\infty$, then $f\in
%X$, and
%\begin{equation}\label{fatou}
%\|f\|_X\le \liminf_{n\to \infty}\|f_n\|_X.
%\end{equation}

\section{Proof of Theorems \ref{Main_thm}, \ref{multbuckley} and \ref{TwoWeights_Thm}}
In the proof of Theorem \ref{Main_thm} we shall use the following reverse H\"older property of $A_{\infty}$ weights proved in \cite{HP}:
if $w\in A_{\infty}$, then
\begin{equation}\label{AInfty_RHI}
\left(\frac{1}{|Q|}\int_Qw^{r(w)}\right)^{1/r(w)}\leq 2 \frac{1}{|Q|}\int_{Q} w,
\end{equation}
where $r(w)= 1+\frac{1}{\tau_n [w]_{A_{\infty}}}$ and $\tau_{n}=2^{11+n}$. Observe that $r'(w)\approx [w]_{A_{\infty}}$.

\begin{proof}[Proof of Theorem \ref{Main_thm}]
By (\ref{estd}), it suffices to prove the theorem for the dyadic maximal operators ${\mathcal M}^{{\mathscr{D}}_{\a}}$.
Since the proof is independent of the particular dyadic grid, without loss of generality we consider
${\mathcal M}^{d}$ taken with respect to the standard dyadic grid ${\mathcal D}$.

Let $a=2^{m(n+1)}$. and $\Omega_k=\{x\in {\mathbb R}^n: {\mathcal M}^d(\vec f\,)(x)>a^k\}$.
We have seen in the proof of Lemma \ref{appr} that $\Omega_k=\cup_jQ_j^k$,
where the family $\{Q_j^k\}$ is sparse and $a^{k}<\prod_{i=1}^m\frac{1}{|Q_j^k|}\int_{Q_j^k}|f_i|\le 2^{nm}a^{k}.$
It follows that
  \begin{eqnarray*}
   && \int_{\Rn} \Mm^d (\vec f)^p \,\nu_{\vec w} dx = \sum_{k} \int_{ \Omega_{k}\setminus\Omega_{k+1}}{\mathcal M}^d(\vec f)^{p}\,\nu_{\vec w} dx \\
    %&\leq a^{p} \sum_{k} a^{kp}\nu_{\vec w}(\Omega_{k})=a^p\sum_{k,j}a^{kp} \nu_{\vec w}( Q_{k,j}) \\
    &&\leq a^p\sum_{k,j} \left(\prod_{i=1}^m\frac{1}{|Q_j^k|}\int_{Q_j^k}|f_i|dy_i\right)^{p}\nu_{\vec w}(Q_j^k)\\
    &&\leq a^p\sum_{k,j} \left(\prod_{i=1}^m\frac{1}{|Q_j^k|}\int_{Q_j^k}|f_i|w_i^{\frac{1}{p_i}}w_i^{-\frac{1}{p_i}}dy_i\right)^{p}\nu_{\vec w}(Q_j^k)\\
    &&\leq a^p \sum_{k,j} \prod_{i=1}^m \left(\frac{1}{|Q_j^k|}\int_{Q_j^k}|f_i|^{\alpha_i} w_i^{\frac{\alpha_i}{p_i}}dy_i\right)^{\frac{p}{\alpha_i}} \left(\frac{1}{|Q_j^k|}\int_{Q_j^k}w_i^{-\frac{\alpha_i'}{p_i}}dy_i \right)^{\frac{p}{\alpha_i'}}\nu_{\vec w}(Q_j^k),
  \end{eqnarray*}
  where $\alpha_i=(p_i'r_i)'$ and $r_i$ is the exponent in the sharp reverse H\"older inequality  \eqref{AInfty_RHI} for the weights $\sigma_i$ which are in $A_{\infty}$ for $i=1,\ldots,m$. Applying \eqref{AInfty_RHI} for each $\sigma_i$, we obtain
  \begin{equation*}\begin{split}
     \int_{\Rn} \Mm^d (\vec f)^p \,\nu_{\vec w}dx &\leq a^p \sum_{k,j} \prod_{i=1}^m \left(\frac{1}{|Q_j^k|}\int_{Q_j^k}|f_i|^{\alpha_i} w_i^{\frac{\alpha_i}{p_i}}dy_i\right)^{\frac{p}{\alpha_i}} \\ &\times\left(2\frac{1}{|Q_j^k|}\int_{Q_j^k}\sigma_i\right)^{\frac{p}{p_i'}}\nu_{\vec w}(Q_j^k) \\
     &\leq C[\vec w]_{A_{\vec P}} \sum_{k,j}\prod_{i=1}^m \left(\frac{1}{|Q_j^k|}\int_{Q_j^k}|f_i|^{\alpha_i}w_i^{\frac{\alpha_i}{p_i}}dy_i\right)^{\frac{p}{\alpha_i}}|Q_j^k|.
  \end{split}\end{equation*}

Let $E_j^k$ be the sets associated with the family $\{Q_j^k\}$.
Using the properties of $E_j^k$ and H\"older's inequality with the exponents $p_i/p$, we get
  \begin{equation*}\begin{split}
    \int_{\Rn} \Mm^d (\vec f)^p \,\nu_{\vec w} dx &\leq 2C [\vec w]_{A_{\vec P}} \sum_{k,j}\prod_{i=1}^m \left(\frac{1}{|Q_j^k|}\int_{Q_j^k}|f_i(y_i)|^{\alpha_i}w_i^{\frac{\alpha_i}{p_i}}dy_i\right)^{\frac{p}{\alpha_i}}|E_j^k| \\
  &\leq 2C[\vec w]_{A_{\vec P}}\sum_{k,j}\int_{E_j^k}\prod_{i=1}^m M\left(|f_i|^{\alpha_i}w_i^{\frac{\alpha_i}{p_i}}\right)^{\frac{p}{\alpha_i}}dx \\
     &\leq 2 C[\vec w]_{A_{\vec P}} \int_{\Rn}\prod_{i=1}^m M\left(|f_i|^{\alpha_i}w_i^{\frac{\alpha_i}{p_i}}\right)^{\frac{p}{\alpha_i}}dx \\
    &\leq 2 C[\vec w]_{A_{\vec P}} \prod_{i=1}^m \left(\int_{\Rn} M\left(|f_i|^{\alpha_i}w_i^{\frac{\alpha_i}{p_i}}\right)^{\frac{p_i}{\alpha_i}}dx\right)^{\frac{p}{p_i}}.\\
    \end{split}\end{equation*}
From this and by the boundedness of $M$,
  \begin{equation*}\begin{split}
  \int_{\Rn} \Mm^d (\vec f)^p \,\nu_{\vec w} dx  &\leq C [\vec w]_{A_{\vec P}} \prod_{i=1}^m \big((p_i/\alpha_i)'\big)^{\frac{p}{p_i}}\Big\| |f_i|^{\alpha_i}w_i^{\frac{\alpha_i}{p_i}}\Big\|^{\frac{p}{\alpha_i}}_{L^{\frac{p_i}{\alpha_i}}(\Rn)}\\
    &\leq C[\vec w]_{A_{\vec P}} \prod_{i=1}^m (p_i' r_i')^{\frac{p}{p_i}}\|f_i\|_{L^{p_i}(w_i)}^p  \\
    &\leq C [\vec w]_{A_{\vec P}} \prod_{i=1}^m ([\sigma_i]_{A_{\infty}})^{\frac{p}{p_i}} \|f_i\|_{L^{p_i}(w_i)}^p,
  \end{split}\end{equation*}
  where in next to last inequality we have used that $(p_i/\alpha_i)'\leq p_i'r_i'$ and in the last inequality we have used that $r_i'\approx [\sigma_i]_{A_{\infty}}$, for $i=1,\ldots,m$. This completes the proof of (\ref{OneWeightBestConstant}).

  Let us show now the sharpness of the exponents in this inequality.
Assume that $n=1$ and $0<\ve<1$. Let
  $$w_i(x)=|x|^{(1-\ve)(p_i-1)}\quad\text{and}\quad f_i(x)=x^{-1+\ve}\chi_{(0,1)}(x),\hspace{1em} i=1,\ldots,m.$$
  It is easy to check that $\nu_{\vec{w}} = |x|^{(1-\ve)(pm-1)}$,
  \begin{equation}\label{Constante}
    [\vec{w}]_{A_{\vec{P}}} = [\nu_{\vec{w}}]_{A_{pm}}\approx (1/\ve)^{mp-1}\quad\text{and}\quad [\sigma_i]_{A_{\infty}}\le \frac{C}{\ve}.
  \end{equation}
    Also,
  \begin{equation}\label{Funciones}
    \prod_{i=1}^m \|f_i\|_{L^{p_i}(w_i)}=(1/\ve)^{1/p}.
  \end{equation}

Let $f=x^{-1+\ve}\chi_{(0,1)}(x)$. Then the left-hand side of (\ref{OneWeightBestConstant}) can be bounded from below as follows:
\begin{equation}\label{estmathc}
\|{\mathcal M}(\vec f\,)\|_{L^p(\nu_{\vec w})}=\|Mf\|_{L^{pm}(\nu_{\vec{w}})}^{m}\ge (1/\ve)^{m}
\|f\|_{L^{pm}(\nu_{\vec{w}})}^{m}=(1/\ve)^{m+1/p}.
\end{equation}
On the other hand, by (\ref{Constante}) and (\ref{Funciones}), the right-hand side of (\ref{OneWeightBestConstant}) is at most $(1/\ve)^{m+1/p}$.
Since $\ve$ is arbitrary, this shows that the exponents $1/p$ and $1/p_i$ on the right-hand side of (\ref{OneWeightBestConstant}) cannot be replaced by smaller ones.
\end{proof}

In order to get an upper bound for $\a$ in part (ii) of Theorem \ref{multbuckley}, we shall need the following technical lemma. Its proof follows the
same lines as the proof of \cite[Th. 3.6]{LOPTT}.

\begin{lemma}\label{technical_lemma}
Let $1< p_j<\infty,j=1,\dots,m$  and $\frac{1}{p}=\frac{1}{p_1}+\ldots+\frac{1}{p_m}$. If $\vec w\in A_{\vec{P}}$, then
  \begin{equation*}
    [\sigma_j]_{A_{\infty}}\leq C[\vec w]_{A_{\vec P}}^{p_j'/p}.
 \end{equation*}
\end{lemma}

\begin{proof} It was shown in \cite[Th. 3.6]{LOPTT} that if $\vec w\in A_{\vec{P}}$, then $\sigma_j\in A_{mp_j'}$. Our goal now is to check that
\begin{equation}\label{check}
[\sigma_j]_{A_{mp_j'}}\le [\vec w]_{A_{\vec P}}^{p_j'/p}.
\end{equation}
Since $[\sigma_j]_{A_{\infty}}\le C[\sigma_j]_{A_{mp_j'}}$, (\ref{check}) would imply the statement of the lemma.

Fix $1\le j\le m$, and define the numbers
$$q_j=p\Big(m-1+\frac{1}{p_j}\Big)\quad\text{and}\quad q_i=\frac{p_i}{p_i-1}\frac{q_j}{p},i\not=j.$$
Since
$$\sum_{i=1}^m\frac{1}{q_i}=\frac{1}{m-1+1/p_j}\Big(\frac{1}{p}+\sum_{i=1,i\not=j}^m(1-1/p_i)\Big)=1,$$
using H\"older inequality, we obtain
\begin{equation*}\begin{split}
\int_{Q}w_j^{\frac{p}{p_jq_j}}&=\int_{Q}
\Big(\prod_{i=1}^mw_i^{\frac{p}{p_iq_j}}\Big)\Big(\prod_{i=1,i\not=j}^mw_i^{-\frac{p}{p_iq_j}}\Big)\\
&\le
\Big(\int_Q\prod_{i=1}^mw_i^{p/p_i}\Big)^{1/q_j}\prod_{i=1,i\not=j}^m\Big(\int_Q
w_i^{-1/(p_i-1)}\Big)^{1/q_i}.
\end{split}\end{equation*}
From this,
\begin{equation*}\begin{split}
  &\left(\int_Qw_j^{1-p_j'}\right)\left(\int_Qw_j^{\frac{p}{p_jq_j}}\right)^{\frac{q_jp_j}{p(p_j-1)}} \\
  &\leq \left(\int_Qw_j^{1-p_j'}\right)\left[ \left(\int_Q \prod_{i=1}^m w_{i}^{p/p_i}\right)^{1/q_j} \prod_{i=1, i\neq j}^m \left( \int_Q w_i^{1-p_i'}\right)^{1/q_i} \right]^{\frac{q_jp_j}{p(p_j-1)}}\\
  &\leq \left(\int_Qw_j^{1-p_j'}\right)\left[ \left(\int_Q \prod_{i=1}^m w_{i}^{p/p_i}\right) \prod_{i=1, i\neq j}^m \left( \int_Q w_i^{1-p_i'}\right)^{q_j/q_i} \right]^{\frac{p_j'}{p}}.
\end{split}\end{equation*}
Since
\begin{equation*}
  \frac{q_j}{q_i}=\frac{p\left(m-1+\frac{1}{p_j}\right)}{\frac{p_i}{p_i-1}\frac{p\left(m-1+\frac{1}{p_j}\right)}{p}} =\frac{p}{p_i'},
\end{equation*}
we obtain
\begin{equation*}\begin{split}
   &\left(\int_Qw_j^{1-p_j'}\right)\left(\int_Qw_j^{\frac{p}{p_jq_j}}\right)^{\frac{q_jp_j}{p(p_j-1)}} \\
   &\leq \left(\int_Qw_j^{1-p_j'}\right)\left[ \left(\int_Q \prod_{i=1}^m w_{i}^{p/p_i}\right) \prod_{i=1, i\neq j}^m \left(\int_Q w_i^{1-p_i'}\right)^{p/p_i'} \right]^{\frac{p_j'}{p}}\\
   &\leq \left[ \left(\int_Q \prod_{i=1}^m w_{i}^{p/p_i}\right)\prod_{i=1}^m \left(\int_Q w_i^{1-p_i'}\right)^{p/p_i'} \right]^{\frac{p_j'}{p}}.
\end{split}\end{equation*}
%
%\comment{{\color{red}We replaced $[\sigma_j]_{A_{mp_j'}}(Q)$ by $A_{mp_j'}(\sigma_j;Q)$ and $[\vec w]_{A_{\vec P}}(Q)$ by $A_{\vec P}(\vec w;Q)$ in the following equation to make them coincide with the notation defined below.}}
%
Therefore,
\begin{eqnarray*}
A_{mp_j'}(\sigma_j;Q)&=&\left(\frac{1}{|Q|}\int_Qw_j^{1-p_j'}\right)\left(\frac{1}{|Q|}\int_Qw_j^{\frac{p}{p_jq_j}}\right)^{\frac{q_jp_j}{p(p_j-1)}}\\
&\le& \left[ \left(\frac{1}{|Q|}\int_Q \nu_{\vec w}\right)\prod_{i=1}^m \left(\frac{1}{|Q|}\int_Q w_i^{1-p_i'}\right)^{p/p_i'} \right]^{\frac{p_j'}{p}}\\
&=& \left(A_{\vec P}(\vec w;Q)\right)^{p_j'/p},
\end{eqnarray*}
which proves (\ref{check}).
\end{proof}

\begin{proof}[Proof of Theorem \ref{multbuckley}]
We start with part (i). Consider the example given in the proof of Theorem \ref{Main_thm}.
Combining (\ref{Constante}), (\ref{Funciones}) and (\ref{estmathc}) with
$$
\|{\mathcal M}(\vec f\,)\|_{L^{p}(\nu_{\vec w})}\le C[\vec w]_{A_{\vec P}}^{\a}\prod_{j=1}^m\|f_j\|_{L^{p_j}(w_j)},
$$
we obtain $m+1/p\le \a(mp-1)+1/p$ which yields $\a\ge \frac{m}{mp-1}$.

Further, by Theorem \ref{Main_thm} and Lemma \ref{technical_lemma},
$$\a\le \frac{1}{p}+\sum_{i=1}^m\frac{1}{p_i}\frac{p_i'}{p}=\frac{1}{p}\Big(1+\sum_{i=1}^m\frac{1}{p_i-1}\Big).$$
This completes the proof of part (i).

Suppose now that $p_1=p_2=\dots=p_m=r$. Then $p=r/m$, $\nu_{\vec w}=\Big(\prod_{j=1}^mw_j\Big)^{1/m}$.
Denote
$$
A_{\vec P}(\vec w;Q)=\Big(\frac{1}{|Q|}\int_Q\nu_{\vec
w}\Big)\prod_{i=1}^m\Big(\frac{1}{|Q|}\int_Q
\si_i\Big)^{(r-1)/m},
$$
where $\si_i=w_i^{1-r'}$. Set also
$${\mathcal M}_{\vec{\si}}(\vec f\,)(x)=\sup_{Q\ni x}\prod_{i=1}^m\frac{1}{\si_i(Q)}\int_Q|f_i|.$$

We will follow the method of the proof of Buckley's theorem given in \cite{Ler3}. By (\ref{estd}), without loss of generality
we may assume that the maximal operators considered below are dyadic.
We get
$$
\prod_{i=1}^m\frac{1}{|Q|}\int_Q|f_i|=A_{\vec P}(\vec w;Q)^{\frac{m}{r-1}}\left(\frac{|Q|}{\nu_{\vec w}(Q)}\Big(\prod_{i=1}^m\frac{1}{\si_i(Q)}\int_Q|f_i|\Big)^{\frac{r-1}{m}}\right)^{\frac{m}{r-1}}.
$$
Hence,
$$
{\mathcal M}(\vec f\,)(x)\le [\vec w]_{A_{\vec P}}^{\frac{m}{r-1}}M_{\nu_{\vec w}}\big({\mathcal M}_{\vec{\si}}(\vec f\,)^{\frac{r-1}{m}}\nu_{\vec w}^{-1}\big)(x)^{\frac{m}{r-1}}.
$$
From this, using H\"older's inequality and
the boundedness of the weighted dyadic maximal operator with the implicit constant independent of the weight, we obtain

\begin{eqnarray*}
\|{\mathcal M}(\vec f\,)\|_{L^p(\nu_{\vec w})}&\le& [\vec w]_{A_{\vec P}}^{\frac{m}{r-1}}
\|M_{\nu_{\vec w}}\big({\mathcal M}_{\vec{\si}}(\vec f\,)^{\frac{r-1}{m}}\nu_{\vec w}^{-1}\big)\|_{L^{r'}(\nu_{\vec w})}^{\frac{m}{r-1}}\\
&\le& C[\vec w]_{A_{\vec P}}^{\frac{m}{r-1}}\|{\mathcal M}_{\vec{\si}}(\vec f\,)\|_{L^{r/m}(\nu_{\vec w}^{1-r'})}\\
&\le& C[\vec w]_{A_{\vec P}}^{\frac{m}{r-1}}\prod_{i=1}^m\|M_{\si_i}(f_i\si_i^{-1})\|_{L^r(\si_i)}\\
&\le& C[\vec w]_{A_{\vec P}}^{\frac{m}{r-1}}\prod_{i=1}^m\|f_i\|_{L^r(w_i)}.
\end{eqnarray*}
This proves that $\a\le \frac{m}{r-1}$. But if $p_1=p_2=\dots=p_m=r$, then $\frac{m}{mp-1}=\frac{m}{r-1}$.
Hence, using part (i), we get that $\a=\frac{m}{r-1}$.
\end{proof}

\begin{proof}[Proof of Theorem \ref{TwoWeights_Thm}]
We start exactly as in the proof of Theorem \ref{Main_thm}. It suffices to prove the main result for ${\mathcal M}^{d}$.
Let $\Omega_k=\{x\in {\mathbb R}^n: {\mathcal M}^d(\vec f\,)(x)>a^k\}=\cup_j{Q_j^k},$ where $a=2^{m(n+1)}$.
Then
\begin{eqnarray*}
\int_{\Rn} \Mm^d (\vec f)^p \,u dx\le
a^p\sum_{k,j} \left(\prod_{i=1}^m\frac{1}{|Q_j^k|}\int_{Q_j^k}|f_i|v_i^{\frac{1}{p_i}}v_i^{-\frac{1}{p_i}}dy_i\right)^{p}u(Q_j^k).
\end{eqnarray*}
By the generalized H\"older inequality (\ref{hol}),
$$
\frac{1}{|Q_j^k|}\int_{Q_j^k}|f_i|v_i^{\frac{1}{p_i}}v_i^{-\frac{1}{p_i}}dy_i\le \|f_iv_i^{\frac{1}{p_i}}\|_{X_i',Q_j^k}\|v_i^{-\frac{1}{p_i}}\|_{X_i,Q_j^k}.
$$
Combining this with the previous estimate, using the properties of the sets $E_j^k$ associated with $\{Q_j^k\}$,
and applying H\"older's inequality, we obtain
\begin{eqnarray*}
&&\int_{\Rn} \Mm^d (\vec f)^p \,u dx\le
a^p\sum_{k,j} \left(\prod_{i=1}^m\|f_iv_i^{\frac{1}{p_i}}\|_{X_i',Q_j^k}\|v_i^{-\frac{1}{p_i}}\|_{X_i,Q_j^k}\right)^{p}\frac{u(Q_j^k)}{|Q_j^k|}|Q_j^k|\\
&&\le 2a^pK^p\sum_{k,j}\left(\prod_{i=1}^m\|f_iv_i^{\frac{1}{p_i}}\|_{X_i',Q_j^k}\right)^{p}|E_j^k|
\le 2a^pK^p \Big\|\prod_{i=1}^m M_{X_i'}(f_iv_i^{1/p_i})\Big\|_{L^p}^{p}\\
&&\le 2a^pK^p\prod_{i=1}^m\|M_{X_i'}(f_iv_i^{1/p_i})\|_{L^{p_i}}^{p}\le
2a^pK^p\prod_{i=1}^{m}\|M_{X_i'}\|_{L^{p_i}}^p\|f_i\|_{L^{p_i}(v_i)}^p,
\end{eqnarray*}
which completes the proof.
\end{proof}

\section{Proof of Theorem \ref{multcz}}
Theorem 3.2 from \cite{LOPTT} says that
$M_{\d}^{\#}(T(\vec f\,))(x)\le C {\mathcal M}(\vec f\,)(x)$,
where $M^{\#}$ is the standard Fefferman-Stein sharp function and $M_{\d}^{\#}(f)=M^{\#}(|f|^{\d})^{1/\d}$.
A simple examination of the proof of this result yields the following local mean oscillation estimate for
$T(\vec f\,)$.

\begin{prop}\label{oscsin} For any cube $Q\subset {\mathbb R}^n$,
\begin{equation}\label{loc1}
\o_{\la}(T(\vec f\,);Q)\le c(T,\la,n)\sum_{l=0}^{\infty}\frac{1}{2^{l\e}}\prod_{i=1}^m\left(\frac{1}{|2^lQ|}\int_{2^lQ}|f_i(y)|dy\right).
\end{equation}
\end{prop}

We turn now to the proof of Theorem \ref{multcz}. Combining Proposition \ref{oscsin} and Theorem \ref{decom1} with $Q_0\in {\mathcal D}$, we get
that there exists a sparse family ${\mathcal S}=\{Q_j^k\}\in {\mathcal D}$ such that for a.e. $x\in Q_0$,
\begin{equation}\label{point1}
|T(\vec f\,)(x)-m_{Q_0}(T(\vec f\,))|\le c\Big(\mathcal M(\vec f\,)(x)+\sum_{l=0}^{\infty}\frac{1}{2^{l\e}}{\mathcal T}_{\mathcal S,l}(\vec{|f|}\,)(x)\Big),
\end{equation}
where $c=c(n,T)$ and
$$
{\mathcal T}_{\mathcal S,l}(\vec f\,)(x)=\sum_{j,k}\Big(\prod_{i=1}^m(f_i)_{2^lQ_j^k}\Big)\chi_{Q_j^k}(x).
$$

By the weak type property of the $m$-linear Calder\'on-Zygmund operators \eqref{weaktype}, assuming, for instance, that each $f_i$ is bounded and with compact support, we get $(T(\vec f\,))^*(+\infty)=~0$. Hence,
it follows from (\ref{pro1}) that $|m_Q\big(T(\vec f\,)\big)|\to 0$ as $|Q|\to\infty$. Therefore, letting $Q_0$ to anyone of $2^n$ quadrants and
using Fatou's lemma and (\ref{point1}), we obtain
$$
\|T(\vec f\,)\|_{X}\le c(n,T)\Big(\|\mathcal M(\vec f\,)\|_X+\sum_{l=0}^{\infty}\frac{1}{2^{l\e}}\sup_{{\mathcal S}\in {\mathcal D}}\|{\mathcal T}_{\mathcal S,l}(\vec f\,)\|_{X}\Big).
$$

By Lemma \ref{appr},
$$\|\mathcal M(\vec f\,)\|_X\le c(m,n)\sup_{{\mathscr{D}},{\mathcal S}}\|{\mathcal A}_{{\mathscr{D}},{\mathcal S}}\vec{|f|}\|_{X}.$$
Our goal now is to show that
\begin{equation}\label{eqTl}
\sup_{{\mathcal S}\in {\mathcal D}}\|{\mathcal T}_{\mathcal S,l}(\vec f\,)\|_{X}
\le c(m,n)l\sup_{{\mathscr{D}},{\mathcal S}}\|{\mathcal A}_{{\mathscr{D}},{\mathcal S}}\vec{|f|}\|_{X}.
\end{equation}
This estimate along with the two previous ones would complete the proof.

\subsection{Several auxiliary operators} Assume that $f_i\ge 0$.
Fix ${\mathcal S}=\{Q_j^k\}\in {\mathcal D}$. Applying Proposition \ref{prhp}, we can decompose the cubes $Q_j^k$ into $2^n$ disjoint families $F_{\a}$ such that for any $Q_j^k\in F_{\a}$ there exists
a cube $P_{j,k}^{l,\a}\in {\mathscr{D}}_{\a}$ such that $2^lQ_j^k\subset P_{j,k}^{l,\a}$ and
$\ell_{P_{j,k}^{l,\a}}\le 6\ell_{2^lQ_j^k}$. Hence,
\begin{equation}\label{estTl}
{\mathcal T}_{\mathcal S,l}(\vec f\,)(x)\le 6^{nm}\sum_{\a=1}^{2^n}\sum_{j,k:Q_j^k\in F_{\a}}\Big(\prod_{i=1}^m(f_i)_{P_{j,k}^{l,\a}}\Big)\chi_{Q_j^k}(x).
\end{equation}
Denote
$$
{\mathcal T}_{l,\a}(\vec f\,)(x)=\sum_{j,k}\Big(\prod_{i=1}^m(f_i)_{P_{j,k}^{l,\a}}\Big)\chi_{Q_j^k}(x).
$$

We shall also need the following auxiliary operator
$$\mathscr{M}_{l,\a}(\vec f_{1,\dots,m-1},g)(x)=\sum_{j,k}\Big(\prod_{i=1}^{m-1}(f_i)_{P_{j,k}^{l,\a}}\Big)\Big(
\frac{1}{|P_{j,k}^{l,\a}|}\int_{Q_j^k}g\Big)\chi_{P_{j,k}^{l,\a}}(x).$$
This object appears naturally in the following duality relation:
\begin{equation}\label{dual}
\int_{{\mathbb R}^n}{\mathcal T}_{l,\a}(\vec f\,)g\,dx=\int_{{\mathbb R}^n}\mathscr{M}_{l,\a}(\vec f_{1,\dots,m-1},g)f_m\,dx.
\end{equation}

\begin{lemma}\label{oscM} For any cube $Q\in {\mathscr{D}}_{\a}$,
$$\o_{\la}(\mathscr{M}_{l,\a}(\vec f_{1,\dots,m-1},g);Q)\le c(\la,m,n)lg_Q\prod_{i=1}^{m-1}(f_i)_Q.$$
\end{lemma}

\begin{proof}
Let $Q\in {\mathscr{D}}_{\a}$ and let $x\in Q$. We have
$$
\mathscr{M}_{l,\a}(\vec f_{1,\dots,m-1},g)(x)=\sum_{k,j:P_{j,k}^{m,\a}\subset Q}+\sum_{k,j:Q\subseteq P_{j,k}^{m,\a}}.
$$
The second sum is a constant (denote it by $c$) for $x\in Q$, while the first sum involves only the functions $f_i$ which are supported in $Q$.
We get the following simple estimate:
\begin{equation}\label{Ml}
|\mathscr{M}_{l,\a}(\vec f_{1,\dots,m-1},g)-c|\chi_Q(x)\le \prod_{i=1}^{m-1}M(f_i\chi_Q)(x)\mathscr{T}_l(g\chi_Q)(x),
\end{equation}
where
$$
\mathscr{T}_lg(x)=\sum_{j,k}\Big(\frac{1}{|P_{j,k}^{m,\a}|}\int_{Q_j^k}g\Big)\chi_{P_{j,k}^{m,\a}}(x).
$$

It was proved in \cite[Lemma 3.2]{Ler2} that $\|\mathscr{T}_lg\|_{L^{1,\infty}}\le c(n)l\|g\|_{L^1}$. Using this estimate, the weak
type $(1,1)$ of $M$, and reiterating the well known property of rearrangements, $(fg)^*(t)\le f^*(t/2)g^*(t/2)$,\, $t>0$,  we get using \eqref{Ml} 
\begin{eqnarray*}
&&\o_{\la}(\mathscr{M}_{l,\a}(\vec f_{1,\dots,m-1},g);Q)\le \big(\prod_{i=1}^{m-1}M(f_i\chi_Q)\mathscr{T}_l(g\chi_Q)\big)^*\big(\la|Q|\big)\\
&&\le\prod_{i=1}^{m-1}\big(M(f_i\chi_Q)\big)^*\big(\la|Q|/2^i\big)\big(\mathscr{T}_l(g\chi_Q)\big)^*\big(\la|Q|/2^{m-1}\big)\\
&&\le c(\la,m,n)lg_Q\prod_{i=1}^{m-1}(f_i)_Q,
\end{eqnarray*}
which completes the proof.
\end{proof}

\subsection{Proof of (\ref{eqTl})} By (\ref{estTl}) it is enough to prove (\ref{eqTl}) with ${\mathcal T}_{l,\a}(\vec f\,)$,
for each $\a=1,\dots,2^n$, instead of
${\mathcal T}_{\mathcal S,l}(\vec f\,)$ on the left-hand side.
By the standard limiting argument one can assume that the sum defining ${\mathcal T}_{l,\a}(\vec f\,)$ is finite.
Then the sum defining the corresponding operator $\mathscr{M}_{l,\a}(\vec f_{1,\dots,m-1},g)$ in (\ref{dual}) will be finite too.
This means that the support of $\mathscr{M}_{l,\a}(\vec f_{1,\dots,m-1},g)$ is compact. One can cover it by at most $2^n$ cubes $Q_{\nu}\in {\mathscr{D}}_{\a}$ such that
$$m_{\mathscr{M}_{l,\a}(\vec f_{1,\dots,m-1},g)}(Q_{\nu})=0, \qquad \nu=1,\cdots, 2^n.$$
Applying Theorem \ref{decom1} along with Lemma \ref{oscM}, we get that there exists a
sparse family ${\mathcal S}_{\a}\in {\mathscr{D}}_{\a}(Q_{\nu})$ such that for a.e. $x\in Q_{\nu}$,
\begin{eqnarray*}
&&\mathscr{M}_{l,\a}(\vec f_{1,\dots,m-1},g)(x)\\
&&\le c(m,n)l\Big(\mathcal{M}^{{\mathscr{D}}_{\a}}(\vec f_{1,\dots,m-1},g)(x)+\sum_{Q_j^k\in {\mathcal S}_{\a}}
\Big(\prod_{i=1}^{m-1}(f_i)_{Q_j^k}\Big)g_{Q_j^k}\chi_{Q_j^k}(x)\Big),
\end{eqnarray*}
where
$$
\mathcal{M}^{{\mathscr{D}}_{\a}}(\vec f_{1,\dots,m-1},g)(x)=\sup_{x\in Q\in {\mathscr{D}}_{\a}}\prod_{i=1}^{m-1}(f_i)_Qg_Q.
$$
Applying to this maximal operator the same argument as in the proof of Lemma \ref{appr} and combining with the previous estimate, we get
that there exists two sparse families ${\mathcal S}_{\a,1}$ and ${\mathcal S}_{\a,2}$ from ${\mathscr{D}}_{\a}$ such that
for a.e. $x\in Q_{\nu}$,
$$\mathscr{M}_{l,\a}(\vec f_{1,\dots,m-1},g)(x)\le c(m,n)l\sum_{\kappa=1}^2\sum_{Q_j^k\in {\mathcal S}_{\a,\kappa}}
\Big(\prod_{i=1}^{m-1}(f_i)_{Q_j^k}\Big)g_{Q_j^k}\chi_{Q_j^k}(x).$$
Hence, by H\"older's inequality (\ref{hol}),
\begin{eqnarray*}
\int_{Q_{\nu}}\mathscr{M}_{l,\a}(\vec f_{1,\dots,m-1},g)f_m\,dx&\le& c(m,n)l\sum_{\kappa=1}^2\int_{{\mathbb R}^n}
{\mathcal A}_{{\mathscr{D}_{\a}},{\mathcal S}_{\a,\kappa}}(\vec f\,)g\,dx\\
&\le& 2c(m,n)l\sup_{{\mathscr{D}},{\mathcal S}}\|{\mathcal A}_{{\mathscr{D}},{\mathcal S}}(\vec f\,)\|_{X}\|g\|_{X'}.
\end{eqnarray*}
Summing up over $Q_{\nu}$ and using (\ref{dual}), we get
$$
\int_{{\mathbb R}^n}{\mathcal T}_{l,\a}(\vec f\,)g\,dx\le
2^{n+1}c(m,n)l\sup_{{\mathscr{D}},{\mathcal S}}\|{\mathcal A}_{{\mathscr{D}},{\mathcal S}}(\vec f\,)\|_{X}\|g\|_{X'}.
$$
By (\ref{dua}), taking here the supremum over $g\ge 0$ with $\|g\|_{X'}=1$ gives (\ref{eqTl}) for ${\mathcal T}_{l,\a}(\vec f\,)$, and
therefore the proof is complete.

\section{Proof of Theorem \ref{mult}}
First, we apply Theorem \ref{multcz} with $X=L^{p}(\nu_{\vec w})$ (observe that $p=\frac{m+1}{m}$).
Fix ${\mathcal S}\in{\mathscr{D}}$. Assume that $f_i\ge 0$. By duality
$$
\|{\mathcal A}_{{\mathscr{D}},{\mathcal S}}(\vec{f})\|_{L^{p}(\nu_{\vec w})}
=\sup_{\|g\|_{L^{p'}\big(\nu_{\vec w}^{-1/(p-1)}\big)=1}}\sum_{j,k}\prod_{i=1}^m(f_i)_{Q_j^k}\int_{Q_j^k}g.
$$

%\comment{{\color{red} In the following lines we have replaced $f$ by $f_i$ in the products. We also replaced $\nu_w$ by $\nu_{\vec w}$}}

Observe that by our choice of $p_i$ we have $p/p'_i=1$. Denote
$$
A_{\vec P}(\vec w;Q)=\Big(\frac{1}{|Q|}\int_Q\nu_{\vec
w}\Big)\prod_{i=1}^m\Big(\frac{1}{|Q|}\int_Q
\si_i\Big).
$$
We have
\begin{eqnarray*}
&&\sum_{j,k}\prod_{i=1}^m(f_i)_{Q_j^k}\int_{Q_j^k}g\\
&&=\sum_{j,k}A_{\vec P}(\vec w;Q_j^k)\Big(\prod_{i=1}^m\frac{1}{\si_i(Q_j^k)}\int_{Q_j^k}f_i\Big)
\Big(\frac{1}{\nu_{\vec w}(Q_j^k)}\int_{Q_j^k}g\Big)|Q_j^k|\\
&&\le 2[\vec w]_{A_{\vec P}}\sum_{j,k}\int_{E_j^k}\prod_{i=1}^mM_{\si_i}^{\mathscr{D}}(f_i\si_i^{-1})M_{\nu_{\vec w}}^{\mathscr{D}}
(g\nu_{\vec w}^{-1})dx\\
&&\le 2[\vec w]_{A_{\vec P}}\int_{{\mathbb R}^n}\prod_{i=1}^mM_{\si_i}^{\mathscr{D}}(f_i\si_i^{-1})M_{\nu_{\vec w}}^{\mathscr{D}}
(g\nu_{\vec w}^{-1})dx.
\end{eqnarray*}
Now we apply H\"older's inequality:
\begin{eqnarray*}
&&\int_{{\mathbb R}^n}\prod_{i=1}^mM_{\si_i}^{\mathscr{D}}(f_i\si_i^{-1})M_{\nu_{\vec w}}^{\mathscr{D}}
(g\nu_{\vec{w}}^{-1})dx\\
&&\le \|\prod_{i=1}^mM_{\si_i}^{\mathscr{D}}(f_i\si_i^{-1})\|_{L^p(\nu_{\vec w}^{-(p-1)})}\|M_{\nu_{\vec w}}^{\mathscr{D}}
(g\nu_{\vec w}^{-1})\|_{L^{p'}(\nu_{\vec w})}.
\end{eqnarray*}
First,
$$\|M_{\nu_{\vec w}}^{\mathscr{D}}(g\nu_{\vec w}^{-1})\|_{L^{p'}(\nu_{\vec w})}\le c\|g\|_{L^{p'}\big(\nu_{\vec w}^{-1/(p-1)}\big)}=c.$$
Second, we apply H\"older's inequality with $p_i/p$ and use that $p-1=\frac{1}{p_i-1}$. We get
$$
\|\prod_{i=1}^mM_{\si_i}^{\mathscr{D}}(f_i\si_i^{-1})\|_{L^p(\nu_{\vec w}^{-(p-1)})}\le \prod_{i=1}^m
\|M_{\si_i}^{\mathscr{D}}(f_i\si_i^{-1})\|_{L^{p_i}(\si_i)}\le c\prod_{i=1}^m\|f_i\|_{L^{p_i}(w_i)},
$$
and we are done.

\end{document}